\documentclass{amsart}
\usepackage{a4wide,mathtools}

\newcommand{\defeq}{\coloneqq}
\newcommand{\IN}{\mathbb N}
\newcommand{\IZ}{\mathbb Z}
\newcommand{\IQ}{\mathbb Q}

\newcommand{\Ra}{\Rightarrow}

\newtheorem{theorem}{Theorem}
\newtheorem{lemma}{Lemma}
\newtheorem{claim}{Claim}
\newtheorem{proposition}{Proposition}
\newtheorem{corollary}{Corollary}

\theoremstyle{definition}

\newtheorem{definition}{Definition}

\title{Semiaffine sets in Abelian groups}
\author{Iryna Banakh, Taras Banakh, Maria Kolinko, Alex Ravsky}
\address{I.~Banakh, A.~Ravsky: Institute for Applied Problems of Mechanics and Mathematics of National Academy of Sciences of Ukraine,  Naukova 3b, Lviv, Ukraine}
\email{ibanakh@yahoo.com, alexander.ravsky@uni-wuerzburg.de}
\address{T.~Banakh, M.~Kolinko: Faculty of Mechanics and Mathematics, Ivan Franko National University of Lviv, Ukraine}
\email{t.o.banakh@gmail.com, marybus20@gmail.com}

\begin{document}
\begin{abstract} A subset $X$ of an Abelian group $G$ is called {\em semiaffine} if for every $x,y,z\in X$ the set $\{x+y-z,x-y+z\}$ intersects $X$. We prove that a subset $X$ of an Abelian group $G$ is semiaffine if and only if one of the following conditions holds: (1) $X=(H+a)\cup (H+b)$ for some subgroup $H$ of $G$ and some elements $a,b\in X$; (2) $X=(H\setminus C)+g$ for some $g\in G$, some subgroup $H$ of $G$ and some midconvex subset $C$ of the group $H$. A subset $C$ of a group $H$ is {\em midconvex} if for every $x,y\in C$, the set $\frac{x+y}2\defeq\{z\in H:2z=x+y\}$ is a subset of $C$.
\end{abstract}
\maketitle

In this paper we study the structure of affine and semiaffine sets in Abelian groups.

 {\em All groups in this paper are assumed to be Abelian}.

\begin{definition} A subset $X$ of a group $G$ is called 
\begin{itemize}
\item {\em affine} if $\forall x,y,z\in X\;x+(y-z)\in X$;
\item {\em semiaffine} if $\forall x,y,z\in X\; \big(x+(y-z)\in X\;\vee\; x-(y-z)\in X\big)$.
\end{itemize}
\end{definition}

It is clear that the empty subset of a group is affine, and every affine set in a group is semiaffine.  

Affine sets in groups are algebraic counterparts of the classical geometric notion of an affine set in a vector space. Semiaffine sets are algebraic counterparts of $1$-spherical metric spaces studied in \cite{BBKR1}. A metric space $(X,d)$ is called {\em $1$-spherical} if for every points $a,b,c\in X$ there exists a points $x\in X$ such that $d(c,x)=d(a,b)$. So, $1$-spherical metric spaces satisfy a weak form of the Axiom of Segment Construction, well-known in Axiomatic Foundations of Geometry, see \cite[p.125]{Bolyai}, \cite[p.81]{BS}, \cite[p.119]{Greenberg},\cite[p.82]{Hartshorne} \cite[p.8]{Hilbert}, \cite[p.11]{SST}. Observe that a metric subspace of the real line  is $1$-spherical if and only if it is a semiaffine subset of the additive group of real numbers.
The problem of classification of $1$-spherical subspaces of the real line was posed in \cite{BBKR1} and was resolved there using the main result of this paper, which characterized semiaffine subsets in groups.

The following simple proposition shows that nonempty affine sets are just shifts of subgroups.

\begin{proposition} For a nonempty subset $X$ of a group $G$ the following conditions are equivalent:
\begin{enumerate}
\item $X$ is affine; 
\item $X-x$ is a subgroup of $G$ for every $x\in X$;
\item $X-x$ is a subgroup of $G$ for some $x\in X$.
\end{enumerate}
\end{proposition}

\begin{proof} $(1)\Ra(2)$. Assume that $X$ is affine and take any $x\in X$. For every $y,z\in X$ the affinity of $X$ ensures that $(y-x)-(z-x)=(x+y-z)-x\in X-x$, witnessing that $X-x$ is a subgroup of $G$.
\smallskip

The implication $(2)\Ra(3)$ is trivial.
\smallskip

$(3)\Ra(1)$ If $X-a$ is a subgroup of $G$ for some $a\in X$, then for every $x,y,z\in X$ we have $(x+y-z)-a=(x-a)+(y-a)-(z-a)\in X-a$ and hence $x+y-z\in X$, which means that the set $X$ is affine.
\end{proof}

A characterization of semiaffine sets in groups is more complicated and involves the notion of a midconvex set in a group.

\begin{definition} A subset $X$ of a group $G$ is {\em midconvex} if for every $x,y\in X$ the set $$\frac{x+y}2\defeq \{z\in G:2z=x+y\}$$ is a subset of $X$.
\end{definition}

The main result of this paper is the following characterization of semiaffine sets in groups.

\begin{theorem}\label{t:main} A subset $X$ of a group $G$ is semiaffine if and only if one of the following conditions holds:
\begin{enumerate}
\item $X=(H+a)\cup (H+b)$ for some subgroup $H$ of $G$ and some elements $a,b\in X$;
\item $X=(H\setminus C)+g$ for some $g\in G$, some subgroup $H\subseteq G$ and some midconvex set $C$ in $H$.
\end{enumerate}
\end{theorem}

Theorem~\ref{t:main} will be proved in Section~\ref{s:main}. Now we discuss some implications of Theorem~\ref{t:main} and the characterizations of midconvex sets from \cite{BBKR}.

Let $G$ be a subgroup of the additive group of real numbers. A subset $C\subseteq G$ is called {\em order-convex} in $G$ if for every real numbers $x\le y$ in $C$, the order interval $\{z\in G:x\le z\le y\}$ is a subset of $C$.

The following characterization of midconvex sets in groups was proved in \cite{BBKR}.

\begin{theorem}\label{t:1} A subset $X$ of a group $G$ is midconvex if and only if for every $g\in G$ and $x\in X$, the set $\{n\in\IZ:x+ng\in X\}$ is equal to $C\cap H$ for some order-convex set $C\subseteq \IZ$ and some subgroup $H\subseteq \IZ$ such that the quotient group $\IZ/H$ has no elements of even order.
\end{theorem}

We recall that the {\em order} of an element $g$ of a group $G$ in the smallest positive integer $n$ such that $ng=0$. If $ng\ne 0$ for all $n\in\IN$, then we say that $g$ has {\em infinite order}. A group $G$ is called {\em periodic} if every element of $G$ has finite order. 

The following characterization of midconvex sets in periodic groups was proved in \cite{BBKR}.

\begin{theorem}\label{t:2} A subset $X$ of a periodic group $H$ is midconvex if and only if for every $x\in X$ the set $X-x$ is a subgroup of $H$ such that the quotient group $H/(X-x)$ contains  no elements of even  order.
\end{theorem}

Combining Theorem~\ref{t:main} with Theorem~\ref{t:2}, we obtain the following characterization of semiaffine sets in periodic groups.

\begin{corollary}\label{c:main2} A subset $X$ of a periodic group $G$ is semiaffine if and only if one of the following conditions holds:
\begin{enumerate}
\item $X=(H+a)\cup (H+b)$ for some subgroup $H$ of $G$ and some elements $a,b\in X$;
\item $X=(H\setminus P)+g$ for some $g\in G$ and some subgroups $P\subseteq H$ of $G$ such that the quotient group $H/P$ contains no elements of even order.
\end{enumerate}
\end{corollary}

The following characterization of midconvex sets in subgroups of the group $\IQ$ of rational numbers was proved in \cite{BBKR}.

\begin{theorem}\label{t:3} Let $H$ be a subgroup of $\IQ$. A nonempty set $ X\subseteq H$ is midconvex in $H$ if and only if $X=C\cap(P+x)$ for some order-convex set $C\subseteq\IQ$, some $x\in X$ and some subgroup $P$ of $H$ such that the quotient group $H/P$ contains no elements of even order. 
\end{theorem}

Combining Theorem~\ref{t:main} with Theorem~\ref{t:3}, we obtain the following characterization of semiaffine sets in subgroups of the group $\IQ$.

\begin{corollary}\label{c:main3} A subset $X$ of a subgroup $G$ of the group $\IQ$ is  semiaffine if and only if one of the following conditions holds:
\begin{enumerate}
\item $X=(H+a)\cup (H+b)$ for some subgroup $H$ of $G$ and some elements $a,b\in X$;
\item $X=(H\setminus (P\cap C))+g$ for some $g\in G$, some order-convex set $C$ in $\IQ$ and some subgroups $P\subseteq H$ of $G$ such that the quotient group $H/P$ contains no elements of even order.
\end{enumerate}
\end{corollary}

\section{Some Lemmas}

In this section we prove some lemmas that will be used in the proof of Theorem~\ref{t:main} presented in the next section. By $\IZ$ be denote the additive group of integer numbers and by $\IN$ the set of positive integer numbers.

\begin{lemma}\label{l:1} Let $X$ be a semiaffine subset of a group $G$ and $a\in X-X$ be such that $2a\notin X-X$. Then $X=(H+x)\cup(H+x+a)$ for some subgroup $H$ of $G$ and some $x\in X$.
\end{lemma}

\begin{proof}  Let $C_a$ be the cyclic subgroup generated by the element $a$ in the group $G$. Let $D=\{n\in\IN:na\in X-X\}$. If $D=\{1\}$, then let $H_a$ be the trivial subgroup of $C_a$. If $D\ne\{1\}$, then let $n=\min(D\setminus\{1\})$. If follows from $a\in X-X$ and $2a\notin X-X$ that $n\ge 3$. Let $H_a$ be the cyclic subgroup generated by the element $g=(n+1)a$ in $C_a$.

\begin{claim} For every $x\in X\cap (X-a)$ we have $X\cap (C_a+x)=(H_a+x)\cup(H_a+x+a)$.
\end{claim}

\begin{proof} First we consider the case when $D=\{1\}$. In this case $H_a=\{0\}$ and $(H_a+x)\cup(H_a+x+a)=\{x,x+a\}$. Assuming that $X\cap (C_a+x)\ne\{x,x+a\}$, we can find an integer  number $k\notin\{0,1\}$ such that $x+ka\in X$. It follows from $\{x,x+a\}\subseteq X$ and  $2a\notin X-X$ that $k\notin\{2,3,-1,-2\}$. If $k>0$, then $ka=(ka+x)-x\in X-X$ and hence $k\in D$, which contradicts $D=\{1\}$. If $k<0$, then $-ka=x-(x+ka)\in X-X$ and hence $-k\in D$, which contradicts $D=\{1\}$. In both cases we obtain a contradiction showing that $X\cap (C_a+x)=(H_a+x)\cup(H_a+x+a)$.

Next, assume that $D\ne\{1\}$. In this case the number $n=\min(D\setminus\{1\})$ is well-defined and we can consider the element $g\defeq (n+1)a$. 
 By induction we shall prove that for every integer number $i\ge 0$, the set $X_i=\{jg+x,jg+x+a:j\in\IZ,\;|j|\le i\}$ is a subset of $X$. The set $X_0=\{x,x+a\}$ is the subset of $X$ by the choice of $x\in X\cap(X-a)$. Assume that for some $i\in\IN$ we know that $X_{i-1}\subseteq X$. Then $\{(i-1)g+x,(i-1)g+x+a\}\subseteq X$. Since $na\in X-X$ and $X$ is semiaffine, either $(i-1)g+x+a+na\in X$ or $(i-1)g+x+a-na\in X$. In the second case, we have $(n-1)a=((i-1)g+x)-((i-1)g+x+a-na)\in X-X$, which contradicts the choice of $n=\min (D\setminus \{1\})$. This contradiction shows that $(i-1)g+x+a-na\notin X$ and hence $ig+x=(i-1)g+x+a+na\in X$. Since $a\in X-X$ and $X$ is semiaffine, $ig+x+a\in X$ or $ig+x-a\in X$. Assuming that $ig+x-a\in X$, we obtain that $(n-1)a=g-2a=ig+x-a-((i-1)g+x+a)\in X-X$, which contradicts the choice of $n=\min(D\setminus\{1\})$. This contradiction shows that $ig+x-a\notin X$ and hence $ig+x+a\in X$. Therefore, $\{ig+x,ig+x+a\}\subseteq X$. By analogy we can prove that $\{-ig+x,-ig+x+a\}\subseteq X$. This completes the inductive step.

After completing the inductive construction, we obtain that $$(H_a+x)\cup(H_a+x+a)=\{ig+x,ig+x+a:i\in \IZ\}=\bigcup_{i\in\IN}X_i\subseteq X\cap (C_a+a).$$ Assuming that  $X\cap (C_a+x)\ne (H_a+x)\cup(H_a+x+a)$, we can find an element $$c\in X\cap (C_a+x)\setminus\big((H_a+x)\cup(H_a+x+a)\big).$$ Write $c$ as $c=ma+x$ for some $m\in\IZ$. Then there exists a unique number $i\in\IZ$ such that $i(n+1)+1<m<(i+1)(n+1)$. Then $(m-(i(n+1)+1))a=c-(ig+x+a)\in X-X$ and $1<m-(i(n+1)+1)<n$, which contradicts the choice of $n=\min (D\setminus\{1\})$. This contradiction shows that $X\cap (C_a+x)=(H_a+x)\cup(H_a+x+a)$.
\end{proof}

\begin{claim}\label{cl:4} For every $x,y,z\in X\cap(X-a)$ we have $x+y-z\in X\cap(X-a)$.
\end{claim}

\begin{proof} It follows from $y,z\in X\cap (X-a)$ that $\{y,z,y+a,z+a\}\subseteq X$ and hence $$\{y-z,y+a-z,y-z-a\}\subseteq X-X.$$ We claim that $x+y-z\in X$. Assuming that $x+y-z\notin X$ and using the semiaffinity of $X$, we conclude that  $x-y+z\in X$. Since $x+a\in X$ and $(x+a)+(y-z-a)=x+y-z\notin X$, the semiaffinity of $X$ ensures that $(x+a)-(y-z-a)\in X$. Then $2a=((x+a)-(y-z-a))-(x-y+z)\in X-X$, which contradicts our assumption. This contradiction shows that $x+y-z\in X$. 

If $x+y-z\notin X-a$, then $(x+a)+y-z\notin X$ and hence $(x+a)-(y-z)\in X$ by the semiaffinity of $X$. Since $x+y+a-z\notin X$ and $y+a-z\in X-X$, the semiaffinity of $X$ ensures that $x-(y+a-z)\in X$. Then $2a=((x+a)-(y-z))-(x-(y+a-z))\in X-X$, which contradicts our assumption. This contradiction shows that $x+y-z\in X\cap (X-a)$.
\end{proof}

Since $a\in X-X$, there exists an element $x\in X\cap (X-a)$. Consider the set $H=(X\cap (X-a))-x$ and observe that for every $h,h'\in H$, the elements $y=x+h$ and $z=x+h'$ belong to $X\cap (X-a)$. By Claim~\ref{cl:4}, $x+h-h'=x+(y-x)-(z-x)=x+y-z=X\cap(X-a)$  and hence $h-h'\in H$, which means that $H$ is a subgroup of the group $G$.

\begin{claim} $X=(H+x)\cup (H+x+a)$.
\end{claim}

\begin{proof} Observe that $H+x=X\cap (X-a)\subseteq X$ and $H+x+a=(X\cap (X-a))+a\subseteq X$. On the other hand, for every $y\in X$, the semiaffinity of $X$ ensures that $y+a\in X$ or $y-a\in X$. In the first case, $y\in X\cap (X-a)=H+x$. In the second case, $y-a\in X\cap (X-a)=H+x$ and $y\in H+x+a$.
\end{proof} 
\end{proof}

\begin{lemma}\label{l:2} Let $X$ be a nonempty semiaffine set in a group $G$. The set $X-X$ is a subgroup of $G$ if and only if $\forall a\in X-X\; (2a\in X-X)$.
\end{lemma}

\begin{proof}  The ``only if'' part is trivial. To prove the ``if'' part, assume that $2a\in X-X$ for all $a\in X-X$. To prove that $X-X$ is a subgroup of $G$, it suffices to check that $a-b\in X-X$ for any $a,b\in X-X$. So, fix any $a,b\in X-X$. Two cases are possible.
\smallskip

1. There exists  $x\in X$ such that $\{x-b,x+b\}\subseteq X$. Since $X$ is semiaffine, $x-a\in X$ or $x+a\in X$. If $x-a\in X$, then $a-b=(x-b)-(x-a)\in X-X$. If $x+a\in X$, then $a-b=(x+a)-(x+b)\in X-X$.
\smallskip

2. For all $x\in X$, $\{x-b,x+b\}\not\subseteq X$. To derive a contradiction, assume that $a-b\notin X-X$. Since $X-X=-(X-X)$, $a-b\notin X-X$ implies $b-a\notin X-X$.  Since $b\in X-X$, there exists $x\in X$ such that $x+b\in X$. Our assumption ensures that 
\begin{multline*}
\{x-b,x-b+a,x-a+b,x+a,x+2b-a\}=\\\{x-b,x-(b-a),x+(b-a),(x+b)-(b-a),(x+b)+(b-a)\}\cap X=\emptyset
\end{multline*}
 Since $x+a\notin X$ and $X$ is semiaffine, $x-a\in X$. Since $(x-a)+b\notin X$ and $X$ is semiaffine, $x-a-b\in X$. Since $2b\in X-X$ and $(x-a)+2b\notin X$, the semiaffinity of $X$ ensures that $x-a-2b\in X$ and then $\{(x-a-b)-b,(x-a-b)+b\}\subseteq X$, which contradicts our assumption. This contradiction shows that $a-b\in X-X$.
\end{proof}

\section{Proof of Theorem~\ref{t:main}}\label{s:main}

The ``if'' part of Theorem~\ref{t:main} is proved in the following two lemmas.

\begin{lemma} A subset $X$ of a group $G$ is semiaffine if $X=(H+a)\cup (H+b)$ for some subgroup $H$ of $G$ and some elements $a,b\in X$.
\end{lemma}

\begin{proof} To show that the set $X=(H+a)\cup (H+b)$ is semiaffine, take any points $x,y,z\in X$.  Depending on the location of the points $x,y,z$ in the set $X=(H+a)\cup(H+b)$, we consider eight cases.
\smallskip

1. If $x\in H+a$, $y\in H+a$, and $z\in H+a$, then $x+y-z\in H+a\subseteq X$.

2. If $x\in H+a$, $y\in H+a$ and $z\in H+b$, then $x-y+z\in H+b\subseteq X$.

3. If $x\in H+a$, $y\in H+b$ and $z\in H+a$, then $x+y-z\in H+b\subseteq X$.

4. If $x\in H+a$, $y\in H+b$, and $z\in H+b$, then $x+y-z\in H+a\subseteq X$.

5. If $x\in H+b$, $y\in H+a$ and $z\in H+a$, then $x+y-z\in H+b\subseteq X$.

6. If $x\in H+b$, $y\in H+a$ and $z\in H+b$, then $x+y-z\in H+a\subseteq X$.

7. If $x\in H+b$, $y\in H+b$ and $z\in H+a$, then $x-y+z\in H+a\subseteq X$.

8. If $x\in H+b$, $y\in H+b$ and $z\in H+b$, then $x+y-z\in H+b\subseteq X$.
\end{proof}

\begin{lemma} A subset $X$ of a group $G$ is semiaffine if $X=(H\setminus C)+g$ for some $g\in G$, some subgroup $H$ of $G$ and some midconvex set $C$ in $H$.
\end{lemma}

\begin{proof} Assuming that  $X=(H\setminus C)+g$ is not semiaffine, we can find elements $x,y,z\in X$ such that $\{x+y-z,x-y+z\}\cap X=\emptyset$ and hence $\{x+y-z-g,x-y+z-g\}\cap (X-g)=\emptyset$. Then $\{x-g,y-g,z-g\}\subseteq X-g=H\setminus C\subseteq H$ and hence  $$\{x+y-z-g,x-y+z-g\}=\{(x-g)+(y-g)-(z-g),(x-g)-(y-g)+(z-g)\}\subseteq H\setminus (X-g)=C.$$ Since $2(x-g)=(x+y-z-g)+(x-y+z-g)$, the midconvexity of $C$ in $H$ ensures that $x-g\in C$, which contradicts the choice of $x\in X=(H\setminus C)+g$.
\end{proof}

To prove the ``only if'' part of Theorem~\ref{t:main}, assume that the set $X$ is semiaffine in the group $G$. If for some $g\in X-X$ we have $2g\notin X-X$, then by Lemma~\ref{l:1}, $X=(H+a)\cup (H+a+g)$ for some subgroup $H$ of $G$ and some $a\in X$. It is clear that $b\defeq g+a\in H+g+a\subseteq X$. Therefore the case (1) of Theorem~\ref{t:main} is satisfied. 

Next, assume that $2g\in X-X$ for all $g\in X-X$.  If $X$ is empty, then $X=(H\setminus C)+0$ for the subgroup $H=G$ and the midconvex set $C=H$ in $H$. So, we assume that $X$ is not empty. In this case the set $H\defeq X-X$ is a subgroup of $G$, according to Lemma~\ref{l:2}. Pick any point $g\in X$. We claim that the set $C\defeq H\setminus(X-g)$ is midconvex in the group $H=X-X$. Indeed, take any points $a,b\in C$ and $c\in H$ with $2c=a+b$. Assuming that $c\notin C$, we conclude that $c\in H\setminus C=X-g$. Since $b-c\in H=X-X$, there exist points $y,z\in X$ such that $b-c=y-z$. Since $c+g\in X$, the semiaffinity of $X$ ensures that $c+g+y-z\in X$ or $c+g-y+z\in X$.

If $c+g+y-z\in X$, then $b+g=(c+g)+(b-c)=(c+g)+(y-z)\in X$ and hence $b\in X-g=H\setminus C$, which contradicts the choice of $b\in C$. 

If $c+g-y+z\in X$, then $a+g=2c-b+g=c+g-(b-c)=c+g-y+z\in X$ and hence $a\in X-g=H\setminus C$, which contradicts the choice of $a\in C$.

In both cases we obtain a contradiction showing that $c\in C$, which means that the set $C$ is midconvex in $H$. Since $X=(H\setminus C)+g$, the condition (2) of Theorem~\ref{t:main} is satisfied.

\end{document}